\documentclass[12pt]{amsart}
\usepackage{bbold}
\usepackage{dsfont}
\usepackage{amsmath}
\usepackage{amsfonts}
\usepackage{amssymb}
\usepackage{graphicx}

\newtheorem{theo}{Theorem}[section]
\theoremstyle{plain}
\newtheorem*{acknowledgment}{Acknowledgment}

\newtheorem{cor}[theo]{Corollary}

\newtheorem{example}[theo]{Example}

\newtheorem{lemma}[theo]{Lemma}

\newtheorem{proposition}[theo]{Proposition}
\newtheorem{remark}[theo]{Remark}
\newtheorem{remarks}[theo]{Remarks}

\numberwithin{equation}{section}
                    
\begin{document}
\title{Preservation of uniform continuity under pointwise product}
\author{A. Bouziad and E. Sukhacheva}
\address[]
{D\'epartement de Math\'ematiques,  Universit\'e de Rouen, UMR CNRS 6085, 
Avenue de l'Universit\'e, BP.12, F76801 Saint-\'Etienne-du-Rouvray, France.}
\email{A. Bouziad: ahmed.bouziad@univ-rouen.fr}
\email{ E. Sukhacheva: sirius9113@mail.ru }
 
\begin{abstract} Let $X$ be a uniform space and $U(X)$ the linear space  of real-valued uniformly continuous functions on  $X$. Our main objective is to give a number of properties 
 characterizing the fact that $U(X)$
is stable under pointwise product in case $X$
is a metric space. Some of these characterizations hold in much  more general circumstances. 
\end{abstract}

\maketitle
\section{Introduction and preliminaries}  
Let $(X,\mathcal U)$ be a Hausdorff uniform space and let $U(X)$ (respectively, $U^*(X)$)
stand for the set of all (bounded) real-valued uniformly continuous on $X$, the reals $\mathbb R$ being
equipped with the usual uniformity. In this paper, we are interested in a  description, possibly 
easy-to-handle,
 of the category of uniform
spaces for which the linear space $U(X)$ is a ring, that is, $U(X)$ is closed under pointwise product.
 As the product
of two bounded uniformly continuous functions
is uniformly continuous, if $U(X)=U^*(X)$ then
  $U(X)$ is a ring.  This is also true if $U(X)$ is locally fine (see \cite{Pu} and \cite{ALM}) or if $U(X)$
coincides with $C(X)$, the ring of real continuous functions on $X$; that is, if $(X,\mathcal U)$
belongs to the class of  u-normal uniform spaces in Nagata's sense \cite{Nag}, nowadays called UC-spaces
or    Atsuji spaces. Let us mention that the study of metric UC-spaces
can be traced back at least to 1947 (Doss \cite{Doss}, Monteiro and Peixoto  \cite{MP}). For more information about UC-spaces, we refer to Beer's book \cite{Be0}; see also \cite{Be}, \cite{HN}, \cite{Ku} and the references therein. \par
The problem  of finding intrinsic conditions that characterize
those metric spaces for which $U(X)$  is ring
is explicitly stated by Nadler in \cite{Na}. The present note was originally motivated  by the recent work \cite{CS} of J. Cabello S\'anchez, where Nadler's question 
is solved as follows:  for every metric space $(X,d)$, $U(X)$ is a ring
if and only if  each  set $B\subset X$ which is not Bourbaki bounded in $X$ contains an infinite uniformly isolated subset. Here,
a set $A\subset X$ is said to be {\it Bourbaki bounded in} the uniform space  $(X,\mathcal U)$
if each $f\in U(X)$ is bounded on $A$ and $A$ is    {\it uniformly isolated in} $(X,{\mathcal U})$
if there is $U\in\mathcal U$ such that $U[x]=\{x\}$ for each $x\in A$ ($A$
is then said to be {\it uniformly $U$-isolated}). 
In Section 2, we shall extend Cabello S\'anchez's criterion    to a large
class of uniform spaces which includes arbitrary products
of metric spaces. \par
In an earlier paper, Artico, Le Donne and Moresco \cite{ALM} proved that for any uniform space $(X,\mathcal U)$, $U(X)$ is a ring if and only
if  every $f\in U(X)$ remains uniformly continuous when $\mathbb R$
is endowed with the uniformity  generated by  polynomial dominated continuous functions. In the same
vein,  a very
recent result established by Beer, Garrido and Mero\~no  \cite[Theorem 2.2]{BGM} asserts that a metric space $(X,d)$ is UC-space iff $1/f\in U(X)$ for each never zero function $f\in U(X)$ (i.e., $U(X)$
is inversion closed \cite{Fr}).   Along the same lines,  we  have  that for any
metric space $(X,d)$, $U(X)$ is a ring if and only if for every $f\in U(X)$ and any $g\in C(\mathbb R)$, 
$g\circ f  \in U(X)$ (that is, $(X,d)$ is  $\mathbb R$-fine in  the sense of \cite{FPV}). This is obtained as consequence of the following result (combining Proposition 2.1 and 2.4): $U(X)$ is a ring iff for each $f\in U(X)$ there is $k\geq 0$ such
that the set $\{x\in X: |f(x)|\geq k\}$ is uniformly isolated.
 \par
By Nagata's theorem \cite[Theorem 3]{Nag},  a metric space $(X,d)$ is  a UC-space  iff
the set $X'$ of limit points is compact and for each $\varepsilon>0$, $X\setminus B(X',\varepsilon)$
is uniformly isolated (here, as  usual, $B(X',\varepsilon)$
is the $\varepsilon$-enlargement of $X'$ with respect to the metric $d$).  See also \cite{Isi} (accordingly to \cite[Theorem 1]{At}), \cite{Hue}
(accordingly to \cite{Be1}) and \cite{Ku} (for proofs and more details).
   Motivated by this characterization of UC-spaces, Cabello S\'anchez
conjectured in his paper \cite{CS} that for every metric space $(X,d)$,   $U(X)$ is a ring
if and only if  there is a Bourbaki bounded set $F\subset X$  such that for any $\varepsilon>0$, the
set $X\setminus B(F,\varepsilon)$ is uniformly isolated.  The  ``if'' part
of this conjecture is established in \cite{CS}. 
We prove in Propositions 2.4 and 2.7 below that the following  slight correction of this conjecture turns out to be  true: $U(X)$ is a ring iff there is a Bourbaki bounded  set $F\subset X$ such that for every $\varepsilon>0$, there is $n\in\mathbb N$
such that $X\setminus B^n(F,\varepsilon)$ is uniformly isolated. \par
 In section 3, we show in Theorem 3.5 that Cabello S\'anchez's conjecture becomes, however, true
for any metric space $(X,d)$ in which Bourbaki bounded sets are totally bounded (e.g, if $(X,d)$ is  non-Archimedean). Along the way, we prove that for any metric space $(X,d)$, the following conditions
are equivalent (Corollary 3.6):  (a) $U(X)$ 
is a ring and every Bourbaki bounded set in $X$ is precompact, (b)  there is a precompact set $K\subset X$ such that for $\varepsilon>0$, $X\setminus B(K,\varepsilon)$ is uniformly isolated and (c)
the completion of $(X,d)$ is a UC-space. The class of metric spaces  having a completion
which is a UC-space was studied by Beer \cite{Be3} and  has been deeply investigated     by T. Jain and S. Kundu in their paper \cite{JK}. No less than twenty-nine equivalent characterizations for a metric space to have a UC completion
are presented in \cite{JK}. The equivalence between  (a) and (c) was  recently established
in \cite[Theorem 3.11]{BGM} and (b) appears to be new.\par
We also give in Section 3 a counter-example to the "only if" part of Cabello S\'anchez's conjecture. 
We learned from Professor J. Cabello S\'anchez that  a counter-example has been given by Beer, Garrido and Mero\~no 
in \cite{BGM},
where his result was subsequently described in terms of the coincidence of the bornology of
Bourbaki bounded subsets with a larger bornology. At first sight, the example in \cite{BGM} seems more complicated than the one proposed here.\par 
 In the final part of Section 3, we investigate the  following natural question:  {\it What are the metrizable spaces $X$ that admit a compatible metric $d$
 such $U(X,d)$ is a ring}? We show in Theorem 3.13 that such a metric exists provided that the set of  limit points of $X$ is contained in a closed finitely chainable
 subspace of $X$.
\section{Some properties  of uniform spaces}
 Our goal in this section  is to
 propose various  properties (items (ii) to (xii) below)  and show that  most of them  characterize  the fact that
$U(X)$ is a ring   for every metric spaces $(X,d)$.  Some of these criteria also apply  to a  class of uniform spaces introduced below
by means of a certain $\omega$-length game; this class  is broad enough to include arbitrary Cartesian products of metric spaces.  \par
For our purpose we shall consider uniformities as were
introduced by Weil, instead of Tukey's coverings approach; so  in this paper
uniformities will be systematically manipulated by means of the set $\mathcal U$ of their filters of entourages. Generally, for undefined  concepts about uniform spaces, we refer 
to  Isbell's book
\cite{Is}. Throughout the paper, all considered $U\in\mathcal U$ are assumed to be open and symmetric (that is, $(x,y)\in U$ iff
$(y,x)\in U)$.  
If $U\in\mathcal U$, $F\subset X$ and $n\in\mathbb N$, then $U^n$ stands for the composition $U\circ\cdots\circ U$, $n$ times, and $U[F]$ denotes the set
of $x\in X$ such that $(x,y)\in U$ for some $y\in F$. \par
Recall that a set $B\subset X$ is Bourbaki bounded  in $(X,\mathcal U)$
if for every $f\in U(X)$, $f(B)$ is bounded in $\mathbb R$. If no confusion can arise, then ``Bourbaki bounded'' will be simplified to ``bounded''. 
It is well-known that $B$ is bounded in $X$ iff for each $U\in\mathcal U$, $B$ is {\it $U$-bounded}, that is, there
are $n\in\mathbb N$ and a finite set $F\subset X$ such that $B\subset U^n[F]$, see for instance \cite{He}.
The set $B$ is said to be {\it totally bounded} (or {\it precompact}) if it is always possible to take $n=1$
in this criterion. Let us mention that in the metric context, the Bourbaki bounded subsets reduce
to precompact sets iff each Bourbaki-Cauchy sequence as defined by Garrido and Mero\~no in their
paper \cite{GM} has a Cauchy subsequence. \par

We denote the positive integers by
$\mathbb N$ and $\mathbb R$ stand for the real line with its usual uniformity. Now we give the definitions of the properties we are going to examine for a given uniform space  $(X,\mathcal U)$:
\begin{itemize}
\item[{\rm (i)}] $U(X)$ is a ring,
\item[{\rm (ii)}] for every $f\in U(X)$ and $g\in U^*(X)$, $fg$
is proximally continuous (the definition is given below),
\item[{\rm (iii)}] for every $f\in U(X)$, there is $k\geq 0$
such that $\{x\in X:|f(x)|\geq k\}$ is uniformly isolated in $X$,
\item[{\rm (iv)}] for every $f\in U(X)$, there is $k\geq 0$ such that $f$
is {\it uniformly locally constant on} $I=\{x\in X:|f(x)|\geq k\}$ (i.e., there is $U\in\mathcal U$ such that for every $x\in I$,
 $f(U[x])=\{f(x)\}$),
 \item[{\rm (v)}] $(X,\mathcal U)$ is {\it $\mathbb R$-fine} (that is, for every $f\in U(X)$
 and $g\in C(\mathbb R)$, $g\circ  f\in U(X)$),
\item[{\rm (vi)}] every unbounded set in $X$ contains an infinite uniformly
isolated set,
\item[{\rm (vii)}] for  every unbounded set $B\subset X$, there are  $U\in\mathcal U$ and
  an infinite set $A\subset B$ such that $U[A]\subset B$.
\item[{\rm (viii)}] for every $U\in\mathcal U$, there are a bounded set $B\subset X$ and $n\in\mathbb N$
such that $X\setminus U^n[B]$ is uniformly discrete,
\item[{\rm (ix)}] for every $U\in\mathcal U$, there are a bounded set $B\subset X$ and $n\in\mathbb N$
such that $X\setminus U^n[B]$ is uniformly isolated,
\item[{\rm (x)}] there is a bounded set $B\subset X$ such that for every $U\in\mathcal U$, there
is $n\in\mathbb N$ such that $X\setminus U^n[B]$ is uniformly isolated,
\item[{\rm (xi)}]  for every $U\in\mathcal U$, there is a bounded set $B\subset X$ such
that $X\setminus U[B]$ is uniformly isolated,
 
\item[{\rm (xii)}] there is a bounded subset $B$ of $X$ such that for every $U\in\mathcal U$,
$X\setminus U[B]$ is uniformly isolated.

 \end{itemize}
 \par
 
    Properties (iii)   should be compared with the characterizations of   UC metric spaces given by condition (6) of \cite[Theorem 1]{At}.  Cabello S\'anchez \cite{CS} has  an example of a metric space $(X,d)$  
 such that $U(X)$ is a ring but for which  there is no bounded set $B\subset X$
 such that $X\setminus B$ is uniformly isolated. So, we can not go further by eliminating the enlargement of
 the bounded set $B$ in (xii). As said above, Cabello S\'anchez also
 proved for metric spaces the equivalence (vi) $\Leftrightarrow$ (i),  the
 implication (xii) $\Rightarrow$ (i) and conjectured that, conversely,  (i) implies (xii).
 Note that this conjecture is at least as strong as 
 the equivalence between (xi) and (i). 
 The equivalence between (iv) and (v) is established (at least in one direction) in \cite{GI} for metric spaces.  \par
The following implications are obvious: (i) $\Rightarrow$ (ii), (v) $\Rightarrow$ (i),  (iii) $\Rightarrow$ (iv),
  (x) $\Rightarrow$ (ix), (xii) $\Rightarrow$ (xi) and (xii) $\Rightarrow$ (x). The next tables
  summarizes the link between the statements proved in the remaining of this section and the relationships between the properties (i)-(x).  The implications 
  that are valid in the class
  of $\beta$-defavorable uniform spaces (defined below) are indicated by (*), whereas the one indicated by $(\dagger)$
  holds  mainly in the context of metric spaces. The  remaining ones hold for arbitrary Hausdorff uniform spaces. 
  \vskip 2mm
  \centerline{
  \begin{tabular}{|c|c|c|c|}
  \hline 
  (ii) $\Rightarrow ^{*}$ (iii) & (iii) $\Rightarrow$ (vi) & (vi) $\Leftrightarrow$ (vii) & (iv) $\Rightarrow$ (v)  \\
  \hline
 2.1 & 2.2 & 2.3 &
  2.4 \\
  \hline
\end{tabular}}

\vskip 2mm

\centerline{
\begin{tabular}{|c|c|c|}
  \hline
   (viii) $\Leftrightarrow$ (ix)   & (ix) $\Rightarrow$ (iii) $\Rightarrow$ (vi) $\Rightarrow^{*}$ (ix) & (ix) $\Rightarrow^{\dagger}$ (x) \\
  \hline
2.5 & 2.6  & 2.7  \\
  \hline
\end{tabular}}

\vskip 2mm
As a consequence, the first ten properties  are equivalent
for metric spaces (this is summarized in Theorem 3.1). On the other hand, as said before, conditions (xi)
and (xii) cannot be added to this list (see Example 3.7); however,   the twelve properties turn out to be  equivalent in every metric space for which
Bourbaki bounded sets are totally bounded (Theorem 3.5).\par
 
\par
Let $(X,\mathcal U)$ be a uniform space and $D$ a dense subset of the product space $X\times X$ when $X$
is given the the topology induced by $\mathcal U$. We shall consider the following game ${\mathcal J}(D)$
between tow Players $\alpha$ and $\beta$. Player $\alpha$ is the first to move and gives $W_0\in\mathcal U$,
and the answer of Player $\beta$ is a point $(x_0,y_0)$ in $W_0\cap D$. At stage  $n\in\mathbb N$, Player $\alpha$
chooses $W_n\in\mathcal U$ and then Player $\beta$ gives $(x_n,y_n)\in W_n\cap D$. The game is of length $\omega$, and
Player $\alpha$ is declared to be the winner of the game $(W_n, (x_n,y_n))_{n\in\mathbb N}$ if for
any infinite set $I\subset\mathbb N$ and $U\in\mathcal U$, there are $n,m\in\mathbb N$
such that $\{n,m\}\cap I\not=\emptyset$ and $(x_n,y_m)\in U$. Otherwise Player $\beta$ wins.\par
We say that the uniform space $(X,\mathcal U)$
is {\it $\beta$-defavorable} if there is  a dense set $D\subset X\times X$ such that Player $\beta$ has no winning strategy in the game ${\mathcal J}(D)$.\par
As an illustration, let $(X_i,{\mathcal U}_i)$, $i\in I$, be a family of nonempty uniform spaces and denote by $(X,\mathcal U)$
their uniform product. Choose $a_i\in X_i$ for each $i\in I$, and let
 $D$ be the set of $x\in X$ for which the set $\{i\in I:x_i\not=a_i\}$ is finite. Then
$D$ is dense in $X$. It is not difficult to show that if each $(X_i,{\mathcal U}_i)$
has a countable basis for its uniformity, then Player $\alpha$
has a strategy $\sigma$ in the game $\mathcal J(D\times D)$ such that
for any $\sigma$-compatible game $(U_n,(x_n,y_n))_{n\in\mathbb N}$, the sequence
$(x_n,y_n)_{n\in\mathbb N}$ converges to the diagonal of $X\times X$ (i.e., every $U\in\mathcal U$
contains all but finitely many $(x_n,y_n)$). In particular, any uniform space which is the product
of metric spaces is $\beta$-defavorable.
\par
 Recall that a
 function $f: X\to\mathbb R$ is said to be {\it proximally continuous} if for any $A\subset X$
and any $\varepsilon>0$, there is $U\in\mathcal U$ such that $f(U[A])\subset B(f(A),\varepsilon)$.
Every uniformly continuous function is proximally
continuous. The converse holds if $X$ is a metric space \cite{Ef}.\par
\par
\vskip 2mm
Our first statement makes the connection between (ii) and (iii):

\begin{proposition} Let $(X,\mathcal U)$ be a $\beta$-defavorable uniform space
and $f\in U(X)$. Suppose that  for every $h\in U^*(X)$, the
product function $fh$ is proximally continuous. Then, there is $n\in\mathbb N$ such
that the set $\{x\in X: |f(x)|\geq n\}$ is uniformly isolated in $(X,\mathcal U)$.
\end{proposition} \begin{proof} Let $D\subset X\times X$
be a dense set such that there is no winning strategy for Player $\beta$
in the game ${\mathcal J}(D)$. Suppose, on the contrary,  that   for every $n\in\mathbb N$ and $U\in\mathcal U$,
there are $y,z\in X$
such that $(y,z)\in U$, $|f(y)|>n$ and $y\not=z$. It is possible to choose
such $(y,z)$ in $D$, since $f$ is continuous and $D$ is dense
in $X\times X$.  Consider the following strategy $\sigma$
for Player $\beta$ in the game ${\mathcal J}(D)$: At stage $n\geq 0$, let $U_n$ be the $n$th move 
of Player  $\alpha$ and assume that $\sigma(U_0,\ldots,U_k)$, $k<n$,
has been defined.  Player $\beta$ chooses $(y_n,z_n)\in U_n\cap D$ such  that $|f(y_{n})|> n+1$
and $y_n\not=z_n$. Since $f$ is uniformly continuous, one can assume that  $|f(y_{n+1})| >1+|f(y_{n})|$ and
 $|f(y_n)-f(z_n)|<1/n$. Note that
 $\{y_n:n\in\mathbb N\}\cap \{z_n:n\in\mathbb N\}=\emptyset$. Finally, define
 $\sigma(U_0,\ldots, U_n)=(y_n,z_n)$. \par
Let $(U_n,(y_n,z_n))_{n\in\mathbb N}$ be a winning game for Player $\alpha$ against the strategy $\sigma$. Following an idea from \cite{CS} (see also \cite{At}), for
each $n\in\mathbb N$, put $g(y_n)=1/(n+1)$ and $g(z_n)=0$. In view of how the sequence $(y_n,z_n)_{n\in\mathbb N}$ has been selected,
$g$ is well defined and  uniformly continuous on the subspace
$\{y_n:n\in\mathbb N\}\cup \{z_n:n\in\mathbb N\}$ of $(X,\mathcal U)$. According to 
 Katetov's theorem \cite{Ka}, $g$ has  a uniformly continuous extension $h: X\to [0,1]$. Let $A=\{z_n:n\in\mathbb N\}$. Since  $fh$ is proximally continuous, there is $U\in\mathcal U$  such that $fh(U[A])\subset B(fh(A),1)$. Since $(U_n,(y_n,z_n))_{n\in\mathbb N}$ is a winning game for Player $\alpha$, there  are $n,m\in\mathbb N$
such that $(y_n,z_m)\in U$. Since $h(A)=\{0\}$, it follows that
 $|fh(y_n)|< 1$, hence $|f(y_n)|<n+1$ which is a contradiction. 
 \end{proof}
For each $U\in\mathcal U$, let $I_U=\{x\in X: U[x]=\{x\}\}$ and let $X'$
stand for the set of limit points of $X$. \par
The implication (iii) $\Rightarrow$ (vi) is a consequence
of the following:

\begin{proposition} Let $(X,\mathcal U)$ be a  uniform space such that for every
$f\in U(X)$ there is $k\geq 0$ such that $\{x\in X: |f(x)|\geq k\}$
is uniformly isolated. Let $B\subset X$. Then, $B$ is bounded if and only if
for each $U\in\mathcal U$, $B\cap I_U$ is finite. In particular,  the set $X'$ of limit points of $X$
is bounded in $X$.
\end{proposition} 
\begin{proof} If $B$ is unbounded, then there is  $f\in U(X)$  such that $f(B)$ is unbounded.
Let $k\geq 0$ and $U\in\mathcal U$ be such that the set $I=\{x\in X:|f(x)|\geq k\}$ is uniformly $U$-isolated. Then
$I\subset I_U$ and $B\cap I$ is infinite, hence $B\cap I_U$ is infinite. \par
The converse is obvious, because if $B\cap I_U$  is infinite
 for some $U\in\mathcal U$, then $B\cap I_U$ (hence $B$) is unbounded in $X$.
\end{proof}

 The following corresponds to the equivalence (vi) $\Leftrightarrow$ (vii):
 
\begin{proposition} The following are equivalent for every  uniform space $(X,\mathcal U)$: 
 \begin{itemize}
\item[ {\rm (a)}]
every unbounded set contains an infinite uniformly isolated set,
\item[ {\rm (b)}]   for every unbounded
set $A\subset X$, there are   $U\in\mathcal U$ and an infinite
set $B\subset A$
such that $U[B]\subset A$.
\end{itemize}
\end{proposition}
 \begin{proof} (a) implies (b) obviously. Suppose that (b)
holds and let $A\subset X$ be an unbounded set. There are $f\in U(X)$
and  a sequence $(a_n)_{n\in\mathbb N}\subset A$ such that 
for each $n\in\mathbb N$, $|f(a_{n+1})|>1+|f(a_n)|$. Let $U\in\mathcal U$ be
such that $|f(x)-f(y)|<1$ whenever $(x,y)\in U$. By (b), there are an infinite  $I\subset \mathbb N$
 and $V\subset U$ such that $V[\{a_n:n\in I\}]\subset \{a_n:n\in\mathbb N\}$. Then
 $\{a_n:n\in I\}$ is uniformly $V$-isolated.
 \end{proof} 
 As said above, the fact that (iv) implies (v) is established in \cite{GI}
 for metric spaces. This also follows 
 from the next general fact,  the proof of which uses the following result from \cite{BL}:
 for every compact set $K\subset\mathbb R$, every continuous function $f:\mathbb R\to\mathbb R$
 is {\it strongly  uniformly continuous at}  $K$, that is, for every $\varepsilon>0$, there
 is $\eta>0$ such that $|f(x)-f(y)|\leq\varepsilon$ whenever $|x-y|<\eta$ and $\{x,y\}\cap K\not=\emptyset$. 

\begin{proposition}
 Let $(X,\mathcal U)$ be uniform space, $f\in U(X)$
 and $g\in C(\mathbb R)$. Suppose that for each $U\in{\mathcal U}$,
 there are $B\subset X$ and $n\in\mathbb N$ such that  $f(B)$ is bounded and
 the composition $g\circ f$ is uniformly continuous on $X\setminus U^n[B]$.
 Then $g\circ f\in U(X)$.
 \end{proposition}  
\begin{proof} Let $\varepsilon>0$
and choose $U\in {\mathcal U}$ such that $(x,y)\in U$ implies $|f(x)-f(y)|<\varepsilon$. Let
 $B\subset X$ and $n\in\mathbb N$ be such that $f(B)$ is bounded and $g\circ f$ is
 uniformly continuous on $X\setminus U^n[B]$. Note that $f$ is bounded
on $U^n[B]$, so by the strong uniform continuity of $g$ at $f(U^n[B])$ \cite{BL}, there is $\eta>0$ such 
that $|g(a)-g(b)|<\varepsilon$ for every
$a,b\in  \mathbb R$ satisfying $\{a,b\}\cap f(U^n[B])\not=\emptyset$ and $|a-b|<\eta$.  Let $V_1\in \mathcal U$ be such that $|g(f(x))-g(f(y))|<\varepsilon$
whenever $(x,y)\in V_1$ and $\{x,y\}\subset X\setminus U^n[B]$. 
Since $f$ is uniformly continuous, there is $V_2\in\mathcal U$ such that $|f(x)-f(y)|<\eta$
for every $(x,y)\in V_2$.  Then, for every $(x,y)\in V_1\cap V_2$, we have
$|g(f(x))-g(f(y))|<\varepsilon$.
 \end{proof}
 
Recall that a subset $A$ of a uniform space $(X,\mathcal U)$
is said to be {\it uniformly discrete} if there is $U\in\mathcal U$
such that for each $x\in A$, $U[x]\cap A=\{x\}$. Clearly,
every uniformly isolated set is uniformly discrete but not conversely. However, as stated in the next remark, it is possible to replace ``uniform isolatedness" in (iii)   by the weaker
condition ``uniform discreteness".  Condition (viii)  is derived from (ix)
in the same way. This will bring some simplification in the proof of Proposition 2.6 below. 
 
\begin{remarks} {\rm  Let $(X,\mathcal U)$
 be a uniform space, $B\subset X$ and $n\in\mathbb N$. \par
 1) If the set $I=X\setminus U^n[B]$ is uniformly discrete, then the set $J=X\setminus U^{n+1}[B]$ is uniformly isolated.
 Indeed, 
 let  $V\in\mathcal U$ be such that for every $x\in I$, $I\cap V[x]=\{x\}$. We assume
 that $V\subset U$.  Let $x\in J$
 and $y\in V[x]$. If $y\not=x$, then $y\in U^n[B]$, hence $x\in U^{n+1}[B]$, a contradiction. \par
 2) Similarly, if $f\in U(X)$ and $k\geq 0$ are such that the set $\{x\in X: |f(x)|\geq k\}$
 is uniformly discrete, then for every $\delta>0$, the set  $\{x\in X: |f(x)|\geq k+\delta\}$
 is uniformly isolated.
 }
 \end{remarks}
 
 \begin{proposition}
 Let $(X,\mathcal U)$ be uniform space. Consider the following:
 \begin{itemize}
 \item[{\rm (a)}] For each $U\in{\mathcal U}$,
 there are a bounded set $B\subset X$ and $n\in\mathbb N$ such that $X\setminus U^n[B]$ is uniformly isolated.
 \item[{\rm (b)}] For every $f\in U(X)$, there is $k\geq 0$ such that $\{x\in X: |f(x)|\geq k\}$
 is uniformly isolated.
 \item[{\rm (c)}] Every unbounded set in $X$ contains an infinite uniformly isolated subset.
 \end{itemize} 
 Then {\rm (a)} $\Rightarrow$ {\rm (b)}  $\Rightarrow$ {\rm (c)}. 
 If $(X,\mathcal U)$ is $\beta$-defavorable, then the three conditions are equivalent.
\end{proposition}  
\begin{proof} To show the implication (a) $\Rightarrow$ (b), let $f\in U(X)$
and choose $U\in\mathcal U$ 
such that $(x,y)\in U$ implies $|f(x)-f(y)|<1$.  
Let $B\subset X$ be a bounded set and $n\in\mathbb N$ such that $X\setminus U^n[B]$ is uniformly isolated.
Then $f$ is bounded on $U^n[B]$, so there exists $k\geq 0$ such that $\{x\in X: |f(x)|\geq k\}\subset X\setminus U^n(B)$. Thus $\{x\in X: |f(x)|\geq k\}$
 is uniformly isolated.\par
 Condition (b) implies (c) by Proposition 2.2.\par
Assume now that $(X,\mathcal U)$
is $\beta$-defavorable and let us show  that (c) $\Rightarrow$ (a). Let $D$ be
a dense subset of $X\times X$ such that   $(X,\mathcal U)$ is $\beta$-defavorable for the game ${\mathcal J}(D)$. In view of Proposition 2.2 and Remark 2.5, it suffices to show 
that there are a  finite set $F\subset X$
 and $n\in\mathbb N$ such that $X\setminus U^n[X'\cup F]$ is uniformly discrete, where
 $X'$ is the set of limit points of the space $X$.   To do that, we proceed by contradiction, so
 suppose that this is not possible. We will define a strategy $\sigma$
for Player $\beta$ in the game $\mathcal J(D)$ as follows. Let $U_0$
 be the first move of Player $\alpha$. Let $x_0, y_0\not\in U^2[X']$ with $y_0\in U_0[x_0]$
  such that $x_0\not=y_0$ (we will  soon see  that $x_0$ and $y_0$
  can be chosen in $D$). Put $F_1=\{x_0,y_0\}$
  and $\sigma(U_0)=(x_0,y_0)$.  Let $U_n$ be the $n$th move of Player $\alpha$ and 
  put $F_n=\{x_i:i<n\}\cup\{y_i: i<n\}$. Since $X\setminus U^{n+1}[X'\cup F_n]$
is not uniformly discrete, there is  $(x_n,y_n)\in U_n$ such that
$x_n\not=y_n$ and $x_n, y_n\not\in U^{n+1}[X'\cup F_n]$. Then $x_n$ and $y_n$ are not in the closure
of $U^{n}[X'\cup F_n]$, so we can assume that $x_n,y_n\in D$
and $x_n, y_n\not\in U^n[X'\cup F_n]$.  Moreover, since $(X,\mathcal U)$
is Hausdorff and $x_i,y_i\not\in X'$ for each $i<n$, by modifying  $U_n$ if necessary,
we  assume that  $U_n[x_i]=\{x_i\}$ and $U_n[y_i]=\{y_i\}$ for each $i<n$. Define $\sigma(U_0,\ldots,U_n)=(x_n,y_n)$. \par
Let
$(U_n,(x_n,y_n))_{n\in\mathbb N}$ be a winning game for  Player $\alpha$ against the strategy $\sigma$.
Then, for every infinite set $I\subset\mathbb N$,  the sets $\{x_n:n\in I\}$ and $\{y_n:n\in I\}$ are $U$-unbounded, hence  by (c) there are  an infinite  $J\subset \mathbb N$ and $V\in\mathcal U$
 such that $\{x_n:n\in J\}$ and $\{y_n:n\in J\}$ are uniformly $V$-isolated.  On the other hand, there are $n\in \mathbb N$
 and $m\in\mathbb N$, with $n\in J$
 or $m\in J$, such that $(x_n,y_m)\in V$. Then $y_m=x_n$, thus $m>n$ and
  $x_n\in U_m[x_m]$. Since $U_m[x_n]=\{x_n\}$, it follows  that
 $x_m=x_n$, which is impossible.
\end{proof}
 We now come to the proof of the implication
  (ix) $\Rightarrow$  (x)
 for metric spaces. This is proved in the next assertion in a somewhat  more general framework.

\begin{proposition}
 Let $(X,\mathcal U)$ be a  uniform space such that for every $U\in\mathcal U$,
 there are a bounded set $B\subset X$ and $n\in\mathbb N$ such that $X\setminus U^n[B]$
 is uniformly isolated. Suppose further that there
 is a sequence $(U_n)_{n\in\mathbb N}\subset \mathcal U$ such that:
 \begin{itemize}
 \item[{\rm (1)}] every
 infinite uniformly isolated set in $X$ contains an infinite
 set which is uniformly $U_n$-isolated  for some $n\in\mathbb N$,
 \item[{\rm (2)}] for every bounded set $B\subset X$ and $U\in\mathcal U$,
 there are $n,m\in\mathbb N$  such that the set $I=U_n[B]\setminus U^m[B]$ is uniformly 
 isolated and $U_n[I]\subset I$.
 \end{itemize}
 Then, there is
 a bounded set $F\subset X$ such that for each $U\in\mathcal U$,
 there is $n\in\mathbb N$ such that $X\setminus U^n[F]$ is uniformly
 isolated.
\end{proposition}  

\begin{proof} We suppose without loss of generality that the sequence $(U_n)_{n\in\mathbb N}$ is decreasing. For each 
  $n\in\mathbb N$, let $L_n\subset X$
  be a  finite set and $k_n\in\mathbb N$  such that  $ X\setminus U_n^{k_n}[L_n]$
  is uniformly isolated. We may suppose that each $L_n$
  is minimal in the sense that  for every $x\in L_n$, $U_n[x]\not=\{x\}$. For if it happens that
  $U_n[x]=\{x\}$ for some $x\in L_n$, then $U_n^{k_n}[x]=\{x\}$, hence $X\setminus U_n^{k_n}[L_n\setminus\{x\}]$
  is uniformly isolated, so $x$ could be removed from the finite set $L_n$.  \par
   We claim  that the set $F=\cup_{n\in\mathbb N}L_n$ is bounded in $X$. To prove this, suppose that $F$
  is not bounded in $X$ and choose by Proposition 2.6 and (1)  an infinite set  $C\subset F$ and 
  $p\in\mathbb N$ such that  $ C$  is uniformly $U_p$-isolated. Since $C$ is infinite, there is  $q\geq p$ 
  such that $ C\cap L_q\not=\emptyset$; choose  $x$ in $ C\cap L_q$. 
  Since $U_q\subset U_p $, we have 
   $U_q[x]\subset U_p[x]=\{x\}$, hence $U_q^{k_q}[x]=\{x\}$,  which is in contradiction to the minimality of $L_q$.  Hence $F$
  must be bounded in $X$ as claimed.\par
  To conclude, let $U\in\mathcal U$ and let us show that $X\setminus U^n[F]$
  is uniformly isolated for some $n\in\mathbb N$.
  By (2), there are  $p,q\in\mathbb N$ such that $U_p[F]\subset U^q[F]\cup I$, where $I$ is uniformly
  isolated and  $U_p[I]\subset I$. Since
  $U_p\circ U=U\circ U_p$ ($U$ and $U_p$ being symmetric), we have 
  $U_p^{k_p}[F]\subset U^{qk_p}[F]\cup I$. Since $I$ and $X\setminus U_p^{k_p}[F]$ are uniformly isolated, we obtain
  that $X\setminus U^{qk_p}[F]$ is uniformly isolated too. 
\end{proof}

\begin{remark} {\rm  Let $(X,\mathcal U)$ be a uniform space and 
let $\mathcal I$ be the set of all uniformly isolated subsets of $X$.  The property (iii) in Section 2 suggests to consider the extended pseudonorm
 $\Vert\cdot\Vert:U(X)\to[0,+\infty]$ defined by 
$$||f||=\inf\{\varepsilon\geq 0: \{x\in X:|f(x)|\geq \varepsilon\}\in{\mathcal I}\}.$$ 
Let us note that, in view of Remark 2.5, an equivalent definition of $\Vert\cdot\Vert$ is obtained
if $\mathcal I$ is replaced by the set of uniformly discrete subsets of $X$. \par
It is easy to check that if $f,g\in U(X)$ and if for some $r\geq 0$, the
sets $\{x\in X:|f(x)|\geq r\}$ and $\{x\in X:|g(x)|\geq r\}$
are uniformly isolated, then $fg\in U(X)$. Hence the
subspace $U^\#(X)$ of $U(X)$ given by all  $f\in U(X)$ such that $||f||<+\infty$ is  a ring  on which
$\Vert\cdot\Vert$ is finite (hence a true pseudonorm). Furthermore, one can prove that $U^\#(X)$ is closed in $U(X)$ and complete with respect
to $\Vert \cdot \Vert$.\par
 
Finally, according to Proposition 2.1, if $(X,\mathcal U)$
is $\beta$-defavorable, then $U^\#(X)$ is the maximal 
ring contained in U(X) and containing $U^*(X)$.    } 
 \end{remark}

 \section{The metric case}
 In view of  the results established in Section 2 and since every metric space is $\beta$-defavorable, we have: 
 
\begin{theo} The ten conditions {\rm (i)-(x)}  are equivalent for every metric space.
 \end{theo}
 Nadler proved in \cite[Theorem 5.2]{Na} that every metric
 space $(X,d)$ for which $U(X)$ is a ring and in which every metrically bounded set
 is Bourbaki bounded, is the union of  a Bourbaki bounded in itself subspace
 and a uniformly isolated set. So it  is possible to add condition
 (xii)  and, a fortiori, condition (xi)  in Theorem 3.1 for any metric space in  which every metrically bounded subspace
is Bourbaki bounded in itself. We shall show in Theorem 3.5 below that
the same result holds for metric spaces in which Bourbaki bounded sets are precompact. \par
The proof 
of Theorem 3.5 uses the following  three lemmas.

\begin{lemma} Let $(Y,d)$ be a metric space and $X$
a dense subset of $Y$. If  every Bourbaki bounded set in $(X,d)$ is precompact, then  the same property
holds in $(Y,d)$.
\end{lemma}
\begin{proof} Let $B\subset Y$ and suppose that $B$ is not precompact or, equivalently, that there is a sequence
$(y_n)_{n\in\mathbb N}\subset B$ without any Cauchy subsequence. For each $n\in\mathbb N$,
let $x_n\in X$ be such that $d(x_n,y_n)<1/n$.  Clearly, the sequence $(x_n)_{n\in\mathbb N}$
has no Cauchy subsequence. It follows that the set $A=\{x_n:n\in\mathbb N\}$ is not Bourbaki bounded in $X$, hence
there exists a uniformly continuous function $f: X\to\mathbb R$ which is not bounded on $A$. Let
$g:Y\to\mathbb R$ be the uniformly continuous extension of $f$ to $Y$. Then $g$
is not bounded on $\{y_n:n\in\mathbb N\}$, hence $B$ is not Bourbaki bounded in $Y$.
\end{proof}

\begin{lemma} Let $X$ be a dense subset of the metric space $(Y,d)$. Then $Y$ has the property that
there exists 
 a precompact subset $K$ of $Y$ such that for each $\varepsilon>0$, $Y\setminus B(K,\varepsilon)$
is uniformly isolated iff $X$ has the same property. 
\end{lemma}
\begin{proof} Suppose that $Y$ has such a precompact set $K$.
 Since $X$ is dense in $Y$, for each $n\in\mathbb N$ there is a finite set $F_n\subset X$
 such that $K\subset B(F_n,1/n)$. We suppose that each $F_n$
is minimal, so that $F_n\subset B(K, 1/n)$. Since $K$ is precompact and
 $\cup_{k\geq n}F_k\subset B(K,1/n)$ for each $n\in\mathbb N$,   the set  $F=\cup_{n\in\mathbb N} F_n$
is  precompact. 
It follows that the set $L=(K\cap X)\cup F$ is precompact. To conclude, let $n\in\mathbb N$
and  let us show by contradiction that $X\setminus B(L,1/n)$ is uniformly isolated.
In the opposite case, we could find  two sequences $(a_k)_{k\in\mathbb N}$ and $(b_k)_{k\in\mathbb N}$
in $X\setminus B(L,1/n)$ such that $\lim d(a_k,b_k)=0$ and $a_k\not=b_k$
for each $k\in\mathbb N$. Since $Y\setminus B(K,1/2n)$ is uniformly isolated, there
is $k\in\mathbb N$ such that, say,  $a_k\in B(K,1/2n)$.
Since $K\subset B(F_{2n},1/2n)$, we get
 $a_k\in B(F,1/n)$, which is a contradiction.\par
 Conversely, if  $L$ is a precompact set satisfying the required property for $X$, then the closure $K$
 of $L$ in $Y$ works for $Y$. The straightforward proof is omitted.
\end{proof}  
Following Waterhouse \cite{Wa}, a {\it C-sequence} is a sequence of pairs $(x_n,y_n)_{n\in\mathbb N}$ of distinct
points in a metric space $(X,d)$ such that $\lim d(x_n,y_n)=0$.  C-sequences are considered by
Nadler in \cite{Na}, where the involved sequences $(x_n)_{n\in\mathbb N}$ and  $(y_n)_{n\in\mathbb N}$ 
  are called {\it twin sequences}. It is proved
 in \cite[Lemma 5.3]{Na} that if $U(X)$ is a ring then twin sequences are metrically bounded. When we turn to  Doss's article \cite{Doss}, we see that
a sequence $(x_n)_{n\in\mathbb N}\subset X$ is said to be {\it accessible}  if there is a
sequence $(y_n)_{n\in\mathbb N}\subset X$ disjoint from $(x_n)_{n\in\mathbb N}$ such that $\lim d(x_n,y_n)=0$.   Doss proved in \cite[Theorem I]{Doss}, among other things, that $(X,d)$ is a UC-space   iff every accessible  sequence in $X$
has a convergent subsequence.  
  The following  improvement of Nadler's result
  can be considered as  the counterpart for $U(X)$ to be a ring of Doss's  criterion.  This lemma is not new because once expressed in terms of the so-called isolation functional (see \cite{BGM}), we see that
it corresponds to the very recent result \cite[Theorem 3.9]{BGM}.

\begin{lemma} Let $(X,d)$ be a metric space. Then $U(X)$ is a ring if and only if twin
 sequences in $(X,d)$ are Bourbaki bounded in $X$.
 \end{lemma}
 \begin{proof} Suppose that $U(X)$ is a ring. By Theorem 3.1, condition (iii) is thus satisfied by $X$. Let $(x_n,y_n)_{n\in\mathbb N}$ be  a C-sequence of $X$. We may suppose that both the sets $\{x_n:n\in\mathbb N\}$ and $\{y_n:n\in\mathbb N\}$ are infinite. Since
 $\lim d(x_n,y_n)=0$, for every $\varepsilon>0$, the set $\{n\in\mathbb N:B(x_n,\varepsilon)=\{x_n\}\}$ is finite. 
 It follows from Proposition 2.2 that the sequence  $(x_n)_{n\in\mathbb N}$ is Bourbaki bounded
 in $X$. Clearly, $(y_n)_{n\in\mathbb N}$ is also Bourbaki bounded in $X$.\par
 To show the converse, we proceed by contradiction. So suppose that
 there are $f,g\in U(X)$ such that $fg\not\in U(X)$. Then there are a C-sequence $(x_n,y_n)_{n\in\mathbb N}$ living in $X$ and $\varepsilon>0$ such that for each $n\in\mathbb N$,
 $|fg(x_n)-fg(y_n)|\geq\varepsilon$. In particular,  $fg$ is not uniformly
 continuous on the metric subspace of $X$ given by $A=\{x_n:n\in\mathbb N\}\cup\{y_n:n\in\mathbb N\}$.
 It follows that at least one of the two functions $f$ or $g$ 
 is not bounded on $A$. Consequently, the sequences $(x_n)_{n\in\mathbb N}$
 and $(y_n)_{n\in\mathbb N}$ are not  Bourbaki bounded in $X$.
 \end{proof}
 In view of Theorem 3.1, the following shows that conditions {\rm (i)-(xii)} are equivalent for every metric space in which
 Bourbaki bounded sets are precompact. 

\begin{theo} Let $(X,d)$ be a metric space in which Bourbaki bounded sets 
are precompact. If $U(X)$ is a ring, then there is a precompact set $L\subset X$ such that
for each $\varepsilon>0$, $X\setminus B(L,\varepsilon)$
is uniformly isolated. 
 \end{theo}
 \begin{proof}  Suppose that $U(X)$ is a ring and let 
   $(Y,d)$  be the completion of $(X,d)$.
 Then $U(Y)$ is a ring, since each  $f\in U(X)$ is the restriction to $X$
 of a unique $g\in U(Y)$. Consequently,  by Proposition 2.2 and Lemma 3.2, the closed subspace $Y'$ of $Y$
 is  precompact, hence compact. \par
 Let $\varepsilon>0$ and suppose that $Y\setminus B(Y',\varepsilon)$
 is not uniformly isolated. Then $Y\setminus B(Y',\varepsilon)$
 contains   twin sequences $(x_n)_{n\in\mathbb N}$
 and $(y_n)_{n\in\mathbb N}$. Since condition (iii) is satisfied by the metric space $(Y,d)$ (Theorem 3.1),
 it follows from Lemma 3.4  that the sequence $(x_n)_{n\in\mathbb N}$
  is Bourbaki bounded in $Y$. By Lemma 3.2, twin sequences in $(Y,d)$
 are precompact, hence $(x_n)_{n\in\mathbb N}$ has a cluster point in $Y'$,
 which is impossible. Consequently, $Y\setminus B(Y',\varepsilon)$
 is  uniformly isolated. We can now apply Lemma 3.3 to conclude the proof.
 \end{proof}
    
 To prove   the next corollary 
we shall  make 
use of the ``if'' part of the following    characterization of metric UC-spaces: $(*)$ $(Y,d)$ is a UC-space if and only 
 if 
   there is a compact set $K\subset Y$ such that for every $\varepsilon>0$,
 the set $Y\setminus B(K,\varepsilon)$ is uniformly isolated.     This  natural  statement   does  not  seem  to  appear   in  the  literature,  although it is not difficult
to obtain: the sufficiency follows from the fact that twin sequences must have cluster points in the compact $K$ and the necessity   follows from \cite[Theorem 3]{Nag}.
 
\begin{cor} For any metric space $(X,d)$, the following are equivalent:
 \begin{itemize}
 \item[{\rm (1)}] $U(X)$ is a ring and every Bourbaki bounded set in $X$ is precompact,
 \item[{\rm (2)}] there is a precompact set $L\subset X$
 such that for each $\varepsilon>0$, $X\setminus B(L,\varepsilon)$ is uniformly isolated,
 \item[{(3)}] the completion of $X$ is a UC-space.
 \end{itemize}
 \end{cor}
 \begin{proof} The implication (1) $\Rightarrow$ (2) being established in Theorem 3.5,
 it remains to show that (2) implies (3) and that (3) implies (1). Assume (2) and let $Y$
 stand for the completion of $(X,d)$. Then, by Lemma 3.3 and the above criterion $(*)$,
 $Y$ is a UC-space. 
To show that (3) implies (1),  assume that the completion $Y$ of $X$ is a UC-space. Then $U(X)$ is a ring (since each $f\in U(X)$
has a uniform  extension to $Y$). Let $L$ be a Bourbaki bounded set in $X$. Then for every $\varepsilon>0$, 
the set $L\setminus B(Y',\varepsilon)$, being uniformly isolated and Bourbaki bounded in $X$, must be finite.
Consequently, 
since $Y'$ is precompact,  $L$ is precompact too.
\end{proof}
 We refer the reader to \cite{Be3} and \cite{JK}  where several equivalent characterizations for a metric space to have a UC 
completion are established. As said above for Lemma 3.4, the equivalence between (1) and (3) in Corollary 3.6 has been established  recently  in
 \cite[Theorem 3.11]{BGM}. Condition (2) seems to be new. \par
 The following example shows that it is not possible to add condition (xi)
 in Theorem 3.1 for arbitrary metric spaces. In particular, (i) and (xii) are not equivalent
 for metric spaces, which   disproves Cabello S\'anchez's conjecture mentioned  Section 1.
  
  \begin{example}{\rm 
  Let $C$ be the subspace of $\mathbb R\times \mathbb R$ given by the union
  of $\{0\}\times [0,1]$ and all line segments $C_m=[0,1]\times \{1/m\}$,  $m\in\mathbb N$. We endow $C$  
  with the so-called intrinsic metric $\delta$: the distance between
  two points $x,y\in C$ is the length of
the "shortest path" in $C$ between  these points. For instance, the metric $\delta$ coincides
with the Euclidean metric on   $\{0\}\times [0,1]$ and on all lines $[0,1]\times \{1/m\}$, $m\in\mathbb N$.
However, for example, if  $x=(1,1/3)$ and $y=(0,1/4)$, then $\delta(x,y)= 1+1/3-1/4$.
Now, let $Y$ be the subspace of $(C,\delta)$ given by the union of 
$$\{(0,0)\}\cup \{0\}\times\{1/m:m\in\mathbb N\}$$
and the lines $$\{k/m:0\leq k\leq m\}\times\{1/m\},\, \, m\in\mathbb N.$$
The metric space $(X,d)$ that we are looking for is the hedgehog 
of $Y$ with $\omega$ spins, the basis point being $(0,0)$. More precisely, $X$ is the set of all $(n,y)$, with $y\in Y$ and
$n\in\mathbb N$, where all the points $(n,(0,0))$, $n\in\mathbb N$, are identified to
a single point $0$. More precisely, 
the metric $d$ of $X$ is given by 
\begin{itemize}

\item[{--}]\,  $d(0,(n,y))=\delta((0,0),y)$,
\item[--]\,  $d((n,y),(n,z))=\delta (y,z)$,  
\item[--]\,  $d((n,y),(m,z))=\delta ((0,0),y)+\delta((0,0),z)$\,  if $n\not=m$. 
\end{itemize}
To simplify, the point $(n,(k/m,1/m))$ is  written $(n,k/m,1/m)$ and the ``line'' $L_m$
is defined by $$L_m=\{(n,k/m,1/m): n\in\mathbb N, 0\leq k\leq m\}.$$ 
Then:\par
\vskip 2mm
  1) {\it A subset of $X$ is uniformly isolated if and only if it is contained in the union of finitely many lines $L_m$.} 
  \begin{proof} Clear.\end{proof}
 \vskip 2mm
  2) {\it $U(X)$ is a ring}. \par
 \begin{proof} By Theorem 3.1, it suffices to check that condition (x) from Section 2 is satisfied
 if we take $B=\{0\}$. Let $\varepsilon>0$ and choose a positive integer  $m_0$  such that $1/m_0\leq \varepsilon$.
  Let $m\geq m_0$ and choose $l\in\mathbb N$
  so that $lm_0\leq m<(l+1)m_0$.  We shall prove that $X\setminus B^{m_0+1}(0,\varepsilon)$
  is uniformly isolated. For every $n\in\mathbb N$, we have $$[0,1]\subset \bigcup_{0\leq k<m_0}\big[kl/m,(k+1)l/m\big]$$
  and $$d((n,kl/m,1/m),(n,(k+1)l/m,1/m))=l/m\leq 1/m_0,$$ from which it follows that
  $$\{(n,k/m,1/m): k\leq m\}\subset B^{m_0}((n,0,1/m),\varepsilon)\subset B^{m_0+1}(0,\varepsilon).$$
  Consequently, $L_m\subset B^{m_0+1}(0,\varepsilon)$, hence
  $\cup_{m\geq m_0}L_m\subset B^{m_0+1}(0,\varepsilon)$. It follows then from 1) that $X\setminus B^{m_0+1}(0,\varepsilon)$ is uniformly isolated as claimed.
  \end{proof}
 \vskip 2mm
  3) {\it There is no Bourbaki bounded set $F\subset X$  such
  that $X\setminus B(F,1)$ is uniformly isolated.} 
  \begin{proof}  Let $F\subset X$
  be a Bourbaki bounded set. By 2) and Theorem 3.1, for each $m\in\mathbb N$, there is $k_m\in\mathbb N$ such that  $F\cap L_m\subset \{(n,k/m,1/m): n\leq k_m, 0\leq k\leq m\}$.
  Suppose that for some $m\in\mathbb N$, $X\setminus B(F,1)\subset L_1\cup \ldots\cup L_m$.  Let $p>m$
  and $q> k_{p}$. Then $(q,1,1/p)\not\in \cup_{i\leq m}L_i$, so there is $x\in F$
  such that $d((q,1,1/p),x)<1$. Since  
  $F\cap L_{p}\subset\{(n,k/p,1/p): n\leq k_{p},0\leq k\leq p\}$, the point  
  $x$ belongs to the set $Y=X\setminus \{(q,k/p,1/p): k\leq p\}$, which is impossible
  because  
  $d((q,1,1/p),Y)\geq 1$. It follows now from 1) that $X\setminus B(F,1)$ is not uniformly isolated.
  \end{proof}
  }
  \end{example}
  
 We would like to  conclude with  some  observations about
  the following question: {\it What are the metrizable spaces $X$ that admit a compatible metric $d$
 such that $U(X,d)$ is a ring}? We do not know the full answer to this question, but we will show
  that if  $X$ contains a finitely chainable closed subspace  $L\subset X$ such that $X'\subset L$,
  then there is a compatible metric $d$ on $X$ such
  that $U(X,d)$ is a ring (Theorem 3.13). \par
  Recall that a metrizable space $Y$ is said to be {\it finitely chainable} if there is
  a compatible metric $d$ on $Y$ such that $(Y,d)$ is Bourbaki bounded.
 In what follows, if $Y$ is subset of a metric space $(X,d)$, then
 $(Y,d)$ stands for the metric subspace of $X$.

\begin{lemma} Let $L$ be a bounded subset of a metric space $(X,d)$
 and let $M\subset X$ be such that $L\subset M$. If    $X\setminus M$ is
 uniformly isolated, then $L$ is bounded in $(M,d)$. 
 \end{lemma}
 \begin{proof} Let $\eta>0$ and choose a finite set $F\subset X$ and $n\in\mathbb N $
 such that $L\subset B^n(F,\eta)$. We may suppose that
   $X\setminus M$
 is $\eta$-uniformly isolated. Let $x\in L$. There are $y\in F$ and a finite sequence of distinct
 points $z_1,\ldots,z_n\in X$, 
 with $z_1=x$ and $z_n=y$, such that $d(z_i,z_{i+1})<\eta$ for each $1\leq i<n$. Suppose that $z_i\in M$.
 Then $z_{i+1}\in M$ because otherwise $z_{i+1}=z_i$, since $z_i\in M$, $d(z_i,z_{i+1})<\eta$ and $B(z_{i+1},\eta)=\{z_{i+1}\}$. Since $z_1\in M$, it follows that $\{z_1,\ldots,z_n\}\subset M$. Consequently,
 $L$ is bounded in $(M,d)$.
 \end{proof}
 
\begin{lemma} Let $(X,d)$
 be a metric space and  $M\subset X$. Then $(M,d)$  is bounded 
 if {\rm(}and only if{\rm)} for each $\varepsilon>0$,
 $M$ is bounded in $(B(M,\varepsilon),d)$.
 \end{lemma}
 \begin{proof} Let $\varepsilon>0$
 and put $Y=B(M,\varepsilon/3)$. There are a finite set
 $F\subset X$ and $n\in\mathbb N$ such that $M\subset B_Y^n(F,\varepsilon/3)$, where $B_Y^n(F,\varepsilon/3)$
 is the $n$-iteration of $B_Y(F,\varepsilon/3)$ in the metric space $(Y,d)$. We shall show that
 $M\subset B_M^n(F,\varepsilon)$. For each $z\in B(M,\varepsilon/3)$, select $s(z)\in M$
 such that $d(z,s(z))<\varepsilon/3$. Now let $x\in M$ and choose a finite sequence $z_1,\ldots,z_n\in Y$ such that $x=z_1$, $z_n\in F$ and $d(z_i,z_{i+1})<\varepsilon/3$
 for each $i<n$. Let  $t_1=x$, $t_n=z_n$ and $t_i=s(z_i)$ for $1< i<n$. Then
 $\{t_1,\ldots,t_n\}\subset M$ and $d(t_i,t_{i+1})<\varepsilon$ for each $i<n$. 
 Hence $M\subset B_M^n(F,\varepsilon)$. 
 \end{proof}
 
 Let $(X,d)$ be a metric space and let
  $L\subset X$.  Following Beer \cite{Be}, let
 $\delta_L$ be the metric on $X$ defined by
 $$\delta_L(x,y)=d(x,y)+\max\{d(x,L), d(y,L)\}$$
 when $x\not=y$. Observe that
  for every
 $x\in \overline L$ and $y\in X$, we have $\delta_L(x,y)\leq 2d(x,y)$. Consequently, 
 for every $\varepsilon>0$, $B_d(L,\varepsilon)\subset B_{\delta_L}(L,2\varepsilon)$, and if $X'\subset {\overline L}$ then $d$ and $\delta_L$ are topologically equivalent. Here
and in what follows, $B_d(L,\varepsilon)$ and $B_{\delta_L}(L,\varepsilon)$ stand for the $\varepsilon$-enlargement
of the set $L$ with respect to $d$ and $\delta_L$, respectively.
 
 \begin{lemma}  For every $\eta>0$, $X\setminus B_{\delta_L}(L,\eta)$
 is uniformly $\eta/2$-isolated in $(X,\delta_L)$.  
 \end{lemma}
 \begin{proof} Otherwise,  
 there are $x,y\in X$, with $x \in X\setminus B_{\delta_L}(L,\eta)$, such
 that $0<\delta_L(x,y)<\eta/2$. Then $d(x,L)<\eta/2$, hence there is  $z\in L$
  such that $d(x,z)<\eta/2$. It follows that $\delta_L(x,L)\leq \delta_L(x,z)=d(x,z)+d(x,L)\leq 2d(x,z)<
 \eta$,  a contradiction.
 \end{proof}
  In general, the distances $d$ and $\delta_L$ are not uniformly equivalent on $X$, as the following shows. 
 \begin{proposition} The following are equivalent:
 \begin{itemize}
 \item[{\rm (a)}] For every $\varepsilon>0$, $X\setminus B_d(L,\varepsilon)$
 is uniformly isolated in $(X,d)$. 
 \item[{\rm (b)}] $d$ and $\delta_L$ are uniformly equivalent on $X$.
  \end{itemize}
 \end{proposition}
 \begin{proof} Suppose that (a) holds. Let $\varepsilon>0$. Choose $\eta>0$ such that $\eta<\varepsilon/3$
 and $X\setminus B_d(L,\varepsilon/3)$ is uniformly $\eta$-isolated
 in $(X,d)$. Let $x,y\in X$
 be such that $d(x,y)<\eta$. If $x$ and $y$ are not in $ B_d(L,\varepsilon/3)$, then $x=y$, hence
 $\delta_L(x,y)=0$. If $\{x,y\}\cap B_d(L,\varepsilon/3)\not=\emptyset$, say $x\in B_d(L,\varepsilon/3)$, then $d(x,L)\leq \varepsilon/3$
 and $d(y,L)\leq d(y,x)+\varepsilon/3$, hence $\delta_L(x,y)\leq \varepsilon$.\par
 Conversely,  let  $\varepsilon>0$
 and suppose that $X\setminus B_d(L,\varepsilon)$ is not uniformly isolated in $(X,d)$. Then,
 for each $n\in\mathbb N$, there are $x_n,y_n\in X\setminus B_d(L,\varepsilon)$
 such that $0<d(x_n,y_n)<1/n$. Since $\delta_L(x_n,y_n)\geq \varepsilon$,
 $d$ and $\delta_L$ are not uniformly equivalent.
 \end{proof}
 
 The following answers the natural question of whether  $U(X,\delta_L)$ is a ring in case $L=X'$.  
 \begin{cor} Let $L=X'$. Then $U(X,\delta_L)$ is a ring iff 
 $(X',d)$  is Bourbaki bounded.
 \end{cor}
 \begin{proof} Suppose that $U(X,\delta_L)$ is a ring. Then, by Proposition 2.2,  $X'$ is bounded in $(X,\delta_L)$. By 
 Lemma 3.10, for each $\varepsilon>0$, $X\setminus B_{\delta_L}(X',\varepsilon)$ is uniformly isolated
 in $(X,\delta_L)$, hence by Lemma 3.8, $X'$ is bounded in $(B(X',\varepsilon),\delta_L)$. 
 It follows now from  Lemma 3.9 that $(X',\delta_L)$ is bounded. Since $d=\delta_L$ on $X'$, we obtain
 that $(X',d)$ is bounded.\par
 The converse follows  from Lemma 3.10 and Theorem 3.1.
 \end{proof}
 
 \begin{theo} Let $X$ be a metrizable space with 
 a finitely chainable  closed subspace $L$ such that $X'\subset L$. Then, there is  
 a compatible metric $\delta$ on $X$ such that for every $\varepsilon>0$,
 $X\setminus B_\delta(L,\varepsilon)$ is uniformly isolated . In particular, $U(X,\delta)$
 is a ring. 
 \end{theo}
 \begin{proof} Let  $d_0$ be a compatible metric on the subspace $L$ of $X$
 such that $(L,d_0)$ is bounded.  By Hausdorff theorem \cite{Ha} (see \cite{Hu} for more information), $d_0$
extends to a compatible metric $d$ on $X$. To conclude,
  let $\delta=\delta_L$ be the metric associated to $d$ and $L$ and apply Corollary 3.12.
  \end{proof}

  \acknowledgment{\rm
  
 We are grateful to the referee for her/his thoughtful remarks and observations and to Professor J. Cabello S\'anchez for drawing our attention to  the work 
  of Gerald  Beer, M. Isabel  Garrido and Ana S. Mero\~no \cite{BGM}. 
We would also like to express our thanks to Professor Ana S. Mero\~no for sending us a preprint of that paper.}

\end{document}